\documentclass[a4paper,10pt]{article}
\usepackage{amsmath,amsthm,amssymb}

\numberwithin{equation}{section}

\newcommand{\id}{\mathord{\operatorname{id}}}
\newcommand{\Z}{\mathbb{Z}}
\newcommand{\R}{\mathbb{R}}
\newcommand{\N}{\mathbb{N}}

\newcommand{\rE}{\operatorname{E}}
\newcommand{\rV}{\operatorname{V}}

\newcommand{\Gr}{\mathcal{G}}
\newcommand{\EG}{\rE(\Gr)}
\newcommand{\VG}{\rV(\Gr)}

\newcommand{\Gam}{\Gamma}

\newcommand{\HNN}{\operatorname{HNN}}

\theoremstyle{plain}
\newtheorem{theorem}{Theorem}[section]

\newtheorem{lemma}[theorem]{Lemma}

\theoremstyle{definition}
\newtheorem{definition}[theorem]{Definition}
\newtheorem{example}[theorem]{Example}

\newtheorem{remark}[theorem]{Remark}

\begin{document}

\begin{center}
{\LARGE\bf  Amenable, transitive and faithful actions of groups acting on trees}

\bigskip

{\sc Pierre Fima$^{(1,2)}$
\setcounter{footnote}{1}\footnotetext{Partially supported by the \textit{Agence Nationale de la Recherche} (Grant ANR 2011 BS01 008 01).}
\setcounter{footnote}{2}\footnotetext{Universit\'e Denis-Diderot Paris 7, IMJ, Chevaleret.
    \\ E-mail: pfima@math.jussieu.fr}}
\end{center}

\begin{abstract}
\noindent We study under which condition an amalgamated free product or an HNN-extension over a finite subgroup admits an amenable, transitive and faithful action on an infinite countable set. We show that such an action exists if the initial groups admit an amenable and almost free action with infinite orbits (e.g. virtually free groups or infinite amenable groups). Our result relies on the Baire category Theorem. We extend the result to groups acting on trees. 
\end{abstract}

\section{Introduction}

The notion of amenable groups or, more generally, amenable actions was first introduced by von Neumann \cite{vN29} who proposed to study whether or not, given a group acting on a set $X$, there exists a mean on $X$ invariant by the action (or equivalently a F{\o}lner sequence for the action).

\begin{definition}
An action $\Gamma\curvearrowright X$ of a countable group $\Gamma$ on a countable set $X$ is called \textit{amenable} if there exists a sequence $(C_n)$ of non-empty finite subsets of $X$ such that
$$\frac{|C_n\Delta gC_n|}{|C_n|}\rightarrow 0\quad\text{for all}\,\,\,g\in\Gamma.$$
Such a sequence $(C_n)$ is called a \textit{F{\o}lner} sequence.
\end{definition}

This notion of amenability is different from the one introduced later by Zimmer \cite{Zi84}.

Obviously, every action of an amenable group (i.e. such that the left translation on itself is an amenable action) is amenable. Greenleaf \cite{Gr69} asked for the converse: does the existence of an amenable action of $\Gamma$ implies the amenability of $\Gamma$? To avoid any trivial negative answer, one should assume that the action is faithful and transitive. If the action is free then the converse holds. However van Douwen \cite{vD90} gave a counter example: the free group $\mathbb{F}_2$ admits a faithful, transitive and amenable action.

This leads Glasner and Monod \cite{GM05} to introduce the class $\mathcal{A}$ of countable groups admitting an amenable, transitive and faithful action. Grigorchuk and Nekrashevych \cite{GN05} have constructed a class of amenable, transitive and faithful actions of finitely generated free groups using Schreier graphs. Simultaneously Glasner and Monod \cite{GM05} gave a necessary and sufficient condition for a free product to be in the class $\mathcal{A}$. In particular, they showed that the class $\mathcal{A}$ is closed under free products. They also asked when free products with amalgamations and HNN-extensions are in $\mathcal{A}$.

S. Moon \cite{Mo08}, \cite{Mo09} showed that a free product of finitely generated free groups amalgamated over a cyclic group is in $\mathcal{A}$. She also proved \cite{Mo10} that an amalgamated free product of amenable groups over a finite group as well as an amalgamated free product of a residually finite group with an infinite amenable group over a finite group is in $\mathcal{A}$.

The initial motivation of the present work was to study the case of an HNN-extension $\Gamma={\rm HNN}(H,\Sigma,\theta)$, where $\Sigma$ is a subgroup of $H$ and $\theta\,:\,\Sigma\rightarrow H$ is an injective group homomorphism. Few results are known: Monod and Popa \cite{MP03} showed that $\Gamma\in\mathcal{A}$ whenever $\Sigma=H\in\mathcal{A}$ and S. Moon observed that the Baumslag-Solitar groups are in $\mathcal{A}$.

We say that an action has \textit{infinite orbits} if every orbit is infinite. Our first result is as follows.

\begin{theorem}\label{THMA}
Let $\Gamma_1$, $\Gamma_2$ and $H$ be countable groups and $\Sigma$ be a finite subgroup of  $\Gamma_1$, $\Gamma_2$ and $H$. Let $\theta\,:\,\Sigma\rightarrow H$ be an injective group homomorphism.
\begin{enumerate}
\item If there exists an amenable and faithful action of $H$ on a countable set with infinite orbits and free on $\Sigma$ and $\theta(\Sigma)$ then ${\rm HNN}(H,\Sigma,\theta)\in\mathcal{A}$.
\item If, for $i=1,2$, there exists an amenable and faithful action of $\Gamma_i$ on a countable set with infinite orbits and free on $\Sigma$ then $\Gamma_1*_{\Sigma}\Gamma_2\in\mathcal{A}$.
\end{enumerate}
\end{theorem}

To prove Theorem \ref{THMA} we use the Baire category Theorem. Such an approach has been used for example in \cite{Ep71}, \cite{Di90}, \cite{GM05}, \cite{Mo08}, \cite{Mo09} and \cite{Mo10}.

An action is called \textit{almost free} if every non-trivial group element acts with finitely many fixed points. During the investigation of the HNN-extension case we realized that the class $\mathcal{A}_{\mathcal{F}}$ of countable groups admitting an amenable and almost free action with infinite orbits on a countable set appears naturally. Observe that the class $\mathcal{A}_{\mathcal{F}}$ contains all infinite amenable groups. Moreover, the amenable, transitive and faithful action of $\mathbb{F}_2$ constructed in \cite{vD90} is actually almost free and an obvious adaptation of his construction shows that $\mathbb{F}_n$ admits an amenable, transitive and almost free action on an infinite countable set for all $n\geq 2$. van Douwen also showed that the same conclusion holds for $\mathbb{F}_{\infty}$.

It is easy to check that if $H$ has an amenable and almost free action with infinite orbits and if $H$ is a finite index subgroup of $\Gamma$ then, the induced action is still amenable and almost free with infinite orbits (and also transitive if the original action is). It follows that virtually free groups are in $\mathcal{A}_{\mathcal{F}}$. Moreover, the obstruction to be in $\mathcal{A}$ discovered in \cite[Lemma 4.3]{GM05} is also an obstruction to be in $\mathcal{A}_{\mathcal{F}}$. Namely, let $N\lhd H$ be a normal subgroup such that the pair $(H,N)$ has the relative property $(T)$. If $H\in\mathcal{A}_{\mathcal{F}}$ then $N$ has finite exponent. In particular, $\Z^{2}\rtimes{\rm SL}_2(\Z)\notin\mathcal{A}_{\mathcal{F}}$. The proof of this assertion is an obvious adaptation of the proof of \cite[Lemma 4.3]{GM05}.

We say that a graph is non-trivial if it has at least two edges, $e$ and its inverse edge $\overline{e}$. Our second result is as follows.

\begin{theorem}\label{THMB}
Let $\Gamma$ be a countable group acting without inversion on a non-trivial tree ${\rm T}$ with finite edge stabilizers and finite quotient graph ${\rm T}/\Gamma$. If all the vertex stabilizers are in $\mathcal{A}_{\mathcal{F}}$ then $\Gamma\in\mathcal{A}$.
\end{theorem}

The case when $\Gamma$ acts with finite vertex stabilizers is not covered by Theorem \ref{THMB} but is obvious since $\Gamma$ must be virtually free.

We prove Theorem \ref{THMB} by induction and by using a slightly stronger version of Theorem \ref{THMA} (see Remarks \ref{rmkhnn} and \ref{rmkfree}). A particular case of Theorem \ref{THMB} is the following. Let $\Gamma_1,\Gamma_2\in\mathcal{A}_{\mathcal{F}}$. Then, for all finite common subgroup $\Sigma<\Gamma_1,\Gamma_2$ one has $\Gamma_1*_{\Sigma}\Gamma_2\in\mathcal{A}$. Also, if $H\in\mathcal{A}_{\mathcal{F}}$ then, for all finite subgroup $\Sigma< H$ and all injective morphism $\theta\,:\,\Sigma\rightarrow H$, one has ${\rm HNN}(H,\Sigma,\theta)\in\mathcal{A}$.

The paper is organized as follows. The section $2$ is a preliminary section in which we prove four basic general lemmas which will be used in the paper. In section $3$ we prove the HNN-extension case of Theorem \ref{THMA}. The section $4$ covers the case of an amalgamated free product. Finally, we prove Theorem \ref{THMB} in section $5$.

\textbf{Acknowledgment.} We are grateful to Georges Skandalis for very helpful conversations and remarks, to Soyoung Moon for carefully reading a first version of this paper and to the referee for his or her remarks improving the clarity of the paper.

\section{Generalities}

Our first lemma is certainly well known but we could not find any reference in the literature.

\begin{lemma}\label{InfiniteFolner}
Let $\Gamma\curvearrowright X$ be an amenable action of a countable group. If the action has infinite orbits then there exists a F{\o}lner sequence $(C_n)$ such that $|C_n|\rightarrow\infty$.
\end{lemma}

\begin{proof}
Let $(D_n)$ be a F{\o}lner sequence for $\Gamma\curvearrowright X$. Suppose that $|D_n|\nrightarrow\infty$. By taking a subsequence if necessary we may and will suppose that $(|D_n|)$ is bounded. It follows that, for all $g\in\Gamma$, $|D_n\Delta gD_n|\rightarrow 0$.

Set $K=\cup_{n\in\N} D_n$ and suppose that $K$ is finite. By the pigeonhole principle, there exists $x_0\in K$ such that $I=\{n\in\N\,:\,x_0\in D_n\}$ is infinite. Let $g\in\Gamma$. Since $gD_n=D_n$ for $n$ big enough, we find $n_0\in I$ such that $gx_0\in gD_{n_0}=D_{n_0}$. Hence, $gx_0\in K$ for all $g\in\Gamma$, contradicting the assumption that $x_0$ has infinite orbit. It follows that $K$ is infinite.

Write $\Gamma=\cup^{\uparrow}F_k$, where $F_k$ are finite subsets. For all $k\in\N$, let $n_k,m_k\in\N$ such that $F_kD_n=D_n$ for all $n\geq n_k$ and $\left|\cup_{n=n_k}^{m_k} D_n\right|\geq k$. Define $C_k=\cup_{n=n_k}^{m_k} D_n$. Then $(C_k)$ is a F{\o}lner sequence and $|C_k|\rightarrow\infty$.
\end{proof}

Let $\Gamma\curvearrowright Y$ be an action and $m\in\N$. Observe that \textit{the complement of the large diagonal in} $Y^m$:
$$\{(y_1,\ldots,y_m)\,:\,y_i\neq y_j\,\,\text{for all}\,\,i\neq j\}\subset Y^m$$
is globally invariant under the diagonal action of $\Gamma$ on $Y^m$.

\begin{lemma}\label{largediag}
If $\Gamma$ acts amenably with infinite orbits on $Y$ then $\Gamma$ acts amenably on the complement of the large diagonal in $Y^m$.
\end{lemma}

\begin{proof}
Denote by $X$ the complement of the large diagonal in $Y^m$. By Lemma \ref{InfiniteFolner}, let $(C_n)$ be a F{\o}lner sequence for the action $\Gamma\curvearrowright Y$ such that $|C_n|\rightarrow\infty$. One has
$$|C_n^m\cap X^c|\leq \sum_{1\leq i< j\leq m}|\{(y_1,\ldots,y_m)\in C_n^m\,:\, y_i=y_j\}|=\frac{m(m-1)}{2}|C_n|^{m-1}.$$
Since $|C_n|\rightarrow\infty$ one has $\frac{|C_n^m\cap X^c|}{|C_n^m|}\rightarrow 0$ and $\frac{|C_n^m\cap X|}{|C_n^m|}\rightarrow 1$. Define $D_n=C_n^m\cap X$. Since $|C_n|\rightarrow\infty$ we have $|D_n|\rightarrow\infty$ and we may assume that $D_n\neq\emptyset$ for all $n\in\N$. It is easy to check that $(C_n^m)$ is a F{\o}lner sequence for the action $\Gamma\curvearrowright Y^m$ hence, for all $g\in \Gamma$, one has
\begin{eqnarray*}
\frac{|D_n\Delta gD_n|}{|D_n|}&\leq& \frac{|D_n\Delta C^m_n|}{|D_n|}+\frac{|C^m_n\Delta gC^m_n|}{|D_n|}+\frac{|gC^m_n\Delta gD_n|}{|D_n|}
=2 \frac{|D_n\Delta C^m_n|}{|D_n|}+\left(\frac{|C^m_n|}{|D_n|}\right)\frac{|C^m_n\Delta gC^m_n|}{|C^m_n|}\\
&=&2\frac{|C_n^m\cap X^c|}{|C_n^m\cap X|}+\left(\frac{|C^m_n|}{|C_n^m\cap X|}\right)\frac{|C^m_n\Delta gC^m_n|}{|C^m_n|}
\rightarrow 0.
\end{eqnarray*}\end{proof}

Given an action $\Gamma\curvearrowright X$ and $g\in\Gamma$, we define ${\rm Fix}_X(g)=\{x\in X\,:\, gx=x\}$. If the context is clear, we omit the subscript $X$. 

\begin{lemma}\label{AFtoF}
If $\Gamma\in\mathcal{A}_{\mathcal{F}}$ then, for every finite subset $F\subset\Gamma$ with $1\notin F$, there exists an amenable and almost free action with infinite orbits $\Gamma\curvearrowright X$ on a countable set $X$ such that ${\rm Fix}_X(g)=\emptyset$ for all $g\in F$.
\end{lemma}

\begin{proof}
Let $\Gamma\curvearrowright Y$ be an amenable and almost free action with infinite orbits on a countable set $Y$. Let $m=\text{Max}\{|\text{Fix}_Y(g)|\,:\,g\in F\}+1$ and $X$ be the complement of the large diagonal in $Y^m$. The action of $\Gamma$ on $X$ is obviously almost free, every orbit is infinite and every element of $F$ acts freely. By Lemma \ref{largediag} this action is also amenable.\end{proof}

The following Lemma is inspired by \cite[Lemma 6]{Mo10}.

\begin{lemma}\label{LemSupp}
Let $\Gamma\curvearrowright Y$ be an action of a countable group $\Gamma$ on a countable set $Y$. Define $X=Y\times\N$ and consider the action $\Gamma\curvearrowright X$ given by $g(y,n)=(gy,n)$ for all $g\in\Gamma$ and $(y,n)\in X$. If $\Gamma\curvearrowright Y$ is amenable then, for all sequence $(a_n)$ of real numbers such that $a_n\rightarrow\infty$, there exists a F{\o}lner sequence $(C_n)$ for $\Gamma\curvearrowright X$ and a subsequence $(a_{\varphi(n)})$ such that $\frac{a_{\varphi(n)}}{|C_n|}\rightarrow 1$.
\end{lemma}

\textit{Proof.} Write $\Gamma=\cup^{\uparrow}F_k$, where $(F_k)_k$ is an increasing sequence of finite subsets. It suffices to construct a strictly increasing sequence of integers $(n_k)_k$ and a sequence of non-empty finite subsets $(C_k)_k$ of $X$ such that $C_k$ is a $(\frac{1}{k},F_k)$-F{\o}lner set for all $k\geq 1$ and,
$$1\leq\frac{a_{n_k}}{|C_k|}< 1+\frac{2}{k}\quad\text{for all}\quad k\geq 1.$$
In the sequel, given $x\in\R$, we denote by $[x]$ the unique integer such that $[x]\leq x< [x]+1$. For $k=1$, let $D_1$ be a $(1, F_1)$-F{\o}lner set for the action $\Gamma\curvearrowright Y$. Let $n_1\geq 1$ big enough such that $[a_{n_1}]\geq|D_1|$. By Euclidean division, we write $[a_{n_1}]=q_1|D_1|+r_1$, where $0\leq r_1<|D_1|$ and $q_1\geq 1$. Define $C_1=\sqcup_{n=1}^{q_1} D_1^{(n)}\subset X$, where $D_1^{(n)}=D_1\times\{n\}$. Then, $|C_1|=q_1|D_1|$ and, for all $g\in F_1$,
$$|gC_1\Delta C_1|\leq \sum_{n=1}^{q_1}|gD_1^{(n)}\Delta D_1^{(n)}|\leq q_1|D_1|=|C_1|.$$
Hence, $C_1$ is a $(1,F_1)$-F{\o}lner set. Moreover, 
$$1\leq \frac{[a_{n_1}]}{|C_1|}\leq\frac{a_{n_1}}{|C_1|}<  \frac{[a_{n_1}]+1}{|C_1|}=\frac{q_1|D_1|+r_1+1}{|C_1|}=1+\frac{r_1+1}{q_1|D_1|}<1+\frac{1}{q_1}+\frac{1}{q_1|D_1|}\leq 1+\frac{2}{q_1}\leq 1+2.$$
Now, suppose that, for $k\geq 1$, we have constructed $n_1<\ldots< n_k\in\N^{*}$ and $C_1,\ldots,C_k\subset X$ such that $C_i$ is a $(\frac{1}{i}, F_i)$-F{\o}lner set  and $1\leq\frac{a_{n_i}}{|C_i|}< 1+\frac{2}{i}$ for all $1\leq i\leq k$. Let $D_{k+1}$ be a $(\frac{1}{k+1}, F_{k+1})$-F{\o}lner set for $\Gamma\curvearrowright Y$ and $n_{k+1}>n_k$ big enough such that $[a_{n_{k+1}}]\geq (k+1)|D_{k+1}|$. Write $[a_{n_{k+1}}]=q_{k+1}|D_{k+1}|+r_{k+1}$, where $0\leq r_{k+1}<|D_{k+1}|$ and $q_{k+1}\geq k+1$. Define $C_{k+1}=\sqcup_{n=1}^{q_{k+1}} D_{k+1}^{(n)}\subset X$, where $D_{k+1}^{(n)}=D_{k+1}\times\{n\}$. Then, $|C_{k+1}|=q_{k+1}|D_{k+1}|$ and, for all $g\in F_{k+1}$,
$$|gC_{k+1}\Delta C_{k+1}|\leq \sum_{n=1}^{q_{k+1}}|gD_{k+1}^{(n)}\Delta D_{k+1}^{(n)}|\leq q_{k+1}\frac{|D_{k+1}|}{k+1}=\frac{|C_{k+1}|}{k+1}.$$
Hence, $C_{k+1}$ is a $(\frac{1}{k+1},F_{k+1})$-F{\o}lner set. Moreover, 
$$1\leq \frac{[a_{n_{k+1}}]}{|C_{k+1}|}\leq \frac{a_{n_{k+1}}}{|C_{k+1}|}<  \frac{[a_{n_{k+1}}]+1}{|C_{k+1}|}=\frac{q_{k+1}|D_{k+1}|+r_{k+1}+1}{|C_{k+1}|}<1+\frac{1}{q_{k+1}}+\frac{1}{q_{k+1}|D_{k+1}|}\leq 1+\frac{2}{q_{k+1}}\leq +\frac{2}{k+1}.\hfill{\Box}$$

Let $H<\Gamma$, $H\curvearrowright Y$ and consider the diagonal action $H\curvearrowright Y\times\Gamma$. Define $X=H\setminus (Y\times\Gamma)$. The \textit{induced action} is the action $\Gamma\curvearrowright X$ given by right multiplication (by the inverse element) on the $\Gamma$ part in $X$. The following proposition contains some standard observations on the induced action. We include a proof for the reader's convenience.

\begin{lemma}\label{induction}
Let $H<\Gamma$, $H\curvearrowright Y$ and $\Gamma\curvearrowright X$ the induced action. The following holds.
\begin{enumerate}
\item If $H\curvearrowright Y$ is faithful then $\Gamma\curvearrowright X$ is faithful.
\item If $H\curvearrowright Y$ is amenable then $H\curvearrowright X$ is amenable.
\item If $H\curvearrowright Y$ has infinite orbits and $H$ is almost malnormal \footnote{The subgroup $H$ is called \textit{almost malnormal} in $\Gamma$ if, for all $g\in\Gamma\setminus H$, the set $g^{-1}Hg\cap H$ is finite.} in $\Gamma$ then $H\curvearrowright X$ has infinite orbits.
\item If $K<\Gamma$ is infinite and $gKg^{-1}\cap H$ is finite for all $g\in\Gamma$ then $K\curvearrowright X$ has infinite orbits.
\item Let $(\Sigma_{\epsilon})_{\epsilon\in E}$ be a family of subgroups of $H$ such that $\Sigma_{\epsilon}\curvearrowright Y$ is free for all $\epsilon\in E$. If, for all $g\in\Gamma\setminus H$, $gHg^{-1}\cap H\subset\cup_{\epsilon\in E,\,s\in H}s\Sigma_{\epsilon} s^{-1}$ then, for all  $h\in H$ such that ${\rm Fix}_Y(h)=\emptyset$, one has  ${\rm Fix}_X(h)=\emptyset$.
\item Let $\Sigma<H$ such that $\Sigma\curvearrowright Y$ is free. If $K<\Gamma$ is a subgroup and $\cup_{g\in\Gamma}gKg^{-1}\cap H\subset\cup_{h\in H}h\Sigma h^{-1}$ then $K\curvearrowright X$ is free. 
\end{enumerate}
\end{lemma}

\begin{proof}
For $(y,g)\in Y\times\Gamma$ we denote by $[y,g]$ its class in $X$. Let $t\in\Gamma$. Observe that
\begin{eqnarray}\label{eqind}
{\rm Fix}_X(t)=\{[y,g]\in X\,:\,gtg^{-1}\in H\,\,\text{and}\,\, y\in {\rm Fix}_Y(gtg^{-1})\}.
\end{eqnarray}

$1.$ Suppose that ${\rm Fix}_X(t)=X$. Then, for all $y\in Y$, $[y,1]\in{\rm Fix}_X(t)$ and Equation $(\ref{eqind})$ implies that $t\in H$ and ${\rm Fix}_Y(t)=Y$. Since $H\curvearrowright Y$ is faithful we have $t=1$.

$2.$ Since the map $y\mapsto[y,1]$ is $H$-equivariant (and injective), the action $H\curvearrowright X$ is amenable whenever that action $H\curvearrowright Y$ is.

$3.$ Let $y\in Y$ and $g\in \Gamma$. Let us show that the $H$ orbit of $[y,g]$ is infinite under the hypothesis of $3$.

\textit{Case 1:} $g\in H$. One has, for all $h\in H$, $h[y,g]=[y,gh^{-1}]=[hg^{-1}y,1]$. Since the map $y\mapsto [y,1]$ is injective, the $H$-orbit of $[y,g]$ is infinite for all $g\in H$ whenever  $H\curvearrowright Y$ has infinite orbits.

\textit{Case 2:} $g\in \Gamma\setminus H$. If the $H$-orbit $H[y,g]$ is finite the stabilizer in $H$ of $[y,g]$ must be infinite. However,
$$\{h\in H\,:\,h[y,g]=[y,g]\}\subset\{h\in H\,:\,ghg^{-1}\in H\}=g^{-1}Hg\cap H$$
which is finite since $H$ is almost malnormal in $\Gamma$.

$4.$ The proof of $4$ is similar:  if the $K$-orbit $K[y,g]$ is finite and $K$ is infinite then the stabilizer in $K$ of $[y,g]$ must be infinite. However,
$$\{k\in K\,:\,k[y,g]=[y,g]\}\subset\{k\in K\,:\,gkg^{-1}\in H\}=K\cap g^{-1}Hg$$
which is finite.

$5.$ Let $h\in H\setminus\{1\}$ such that ${\rm Fix}_X(h)\neq\emptyset$ and take $[y,g]\in{\rm Fix}_X(h)$. By Equation $(\ref{eqind})$, $ghg^{-1}\in H$ and $y\in{\rm Fix}_Y(ghg^{-1})$. If $g\in\Gamma\setminus H$ then, there exists $\epsilon\in E$, $\sigma\in\Sigma_{\epsilon}\setminus\{1\}$ and $s\in H$ such that $ghg^{-1}=s\sigma s^{-1}$. Hence, $y\in{\rm Fix}_Y(ghg^{-1})={\rm Fix}_Y(s\sigma s^{-1})=s{\rm Fix}_Y(\sigma)$. Hence, ${\rm Fix}_Y(\sigma)\neq\emptyset$, a contradiction since $\Sigma_{\epsilon}\curvearrowright Y$ is free. Thus, $g\in H$ and $y\in{\rm Fix}_Y(ghg^{-1})=g{\rm Fix}_Y(h)$ which implies that ${\rm Fix}_Y(h)\neq\emptyset$.

$6.$ The proof of $6$ is similar: Let $k\in K\setminus\{1\}$. If $[y,g]\in{\rm Fix}_X(k)$ Equation $(\ref{eqind})$ implies that $gkg^{-1}\in H$ and $y\in{\rm Fix}_Y(gkg^{-1})$. By the hypothesis on $K$, there exists $h\in H$ and $\sigma\in\Sigma\setminus\{1\}$ such that $gkg^{-1}=h\sigma h^{-1}$. Hence,  ${\rm Fix}_Y(gkg^{-1})={\rm Fix}_Y(h\sigma h^{-1})=h{\rm Fix}_Y(\sigma)=\emptyset$, a contradiction.
\end{proof}

\begin{example}\label{example}The following holds.
\begin{enumerate}
\item Let $\Gamma={\rm HNN}(H,\Sigma,\theta)=\langle H,t\,|\,\theta(\sigma)=t\sigma t^{-1},\,\,\forall \sigma\in\Sigma\rangle$ be a non-trivial HNN-extension. The hypothesis of $5$ holds and $H$ is almost malnormal in $\Gamma$ whenever $\Sigma$ is finite. Actually, for all $g\in\Gamma\setminus H$, there exists $s\in H$ such that $gHg^{-1}\cap H\subset s\Sigma s^{-1}$ or $gHg^{-1}\cap H\subset s\theta(\Sigma) s^{-1}$. Indeed, let $v_0$ be the vertex in the Bass-Serre tree of $\Gamma$ with stabilizer $H$ and take $g\notin H$. Denote by $e$ the unique edge on the geodesic $[gv_0,v_0]$ that contains $v_0$. Since $gv_0$ and $v_0$ are fixed by $gHg^{-1}\cap H$ the geodesic $[gv_0,v_0]$ is pointwise fixed by $gHg^{-1}\cap H$. In particular, $gHg^{-1}\cap H$ is contained in the stabilizer of $e$ which is of the form $s\Sigma s^{-1}$ or $s\theta(\Sigma) s^{-1}$ with $s\in H$.

\item If $\Gamma=\Gamma_1*_{\Sigma}\Gamma_2$ is a non-trivial amalgamated free product. By the same geometric argument, for all $i=1,2$, for all $g\in\Gamma\setminus\Gamma_i$, there exists  $s\in\Gamma_i$ such that $g\Gamma_i g^{-1}\cap\Gamma_i\subset s\Sigma s^{-1}$. Hence, the hypothesis of $5$ holds (with $H=\Gamma_i$) and $\Gamma_i$ is malnormal in $\Gamma$ whenever $\Sigma$ is finite. Also, for all $g\in\Gamma$, there exists $s\in\Gamma_1$ such that $g\Gamma_2 g^{-1}\cap \Gamma_1\subset s\Sigma s^{-1}$. Hence, the hypothesis of $4$ (if $\Sigma$ is finite) and $6$ hold with $H=\Gamma_1$ and $K=\Gamma_2$ and, by symmetry, with $H=\Gamma_2$ and $K=\Gamma_1$.\end{enumerate}
\end{example}

\section{HNN-extensions in the class $\mathcal{A}$}

This section is dedicated to the proof of Theorem \ref{THMA}.1.

Let $H$ be a countable group, $\Sigma<H$ a finite subgroup and $\theta\,:\,\Sigma\rightarrow H$ an injective group homomorphism. Define the $\HNN$-extension $\Gam=\HNN(H,\Sigma,\theta)=\langle H,t\,|\,\theta(\sigma)=t\sigma t^{-1},\,\,\forall \sigma\in\Sigma\rangle$.

Let $X$ be an infinite countable set and denote by $S(X)$ the Polish\footnote{With the topology of pointwise convergence: $w_n\rightarrow w$ $\Leftrightarrow$ for all finite subset $F\subset X$ $\exists n_0\in \N$ such that $ w_n|_F=w|_F$ $\forall n\geq n_0$.} group of bijections of $X$. 

For the rest of this section we suppose that $H<S(X)$ such that the actions of $\Sigma$ and $\theta(\Sigma)$ on $X$ are free. Define
$$Z=\{w\in S(X)\,:\,w\sigma w^{-1}=\theta(\sigma)\,\,\text{for all}\,\,\sigma\in\Sigma\}.$$
It is clear that $Z$ is a closed subset of $S(X)$. Moreover, since the actions of $\Sigma$ and $\theta(\Sigma)$ are free, it is easy to see that $Z$ is non-empty.

By the universal property, for all $w\in Z$, there exists a unique group homomorphism $\pi_{w}\,:\,\Gamma\rightarrow S(X)$ such that $\pi_{w}(t)=w$ and $\pi_{w}(h)=h$ for all $h\in H$. The strategy is to prove, under suitable assumptions on the action $H\curvearrowright X$, that the set of $w\in Z$ such that $\pi_{w}$ is amenable, transitive and faithful is a dense $G_{\delta}$ in $Z$.

We first study the set of $w\in Z$ such that $\pi_w$ is transitive.

\begin{lemma}\label{transitif}
If the action $H\curvearrowright X$ has infinite orbits then the set $U=\{w\in Z\,:\,\pi_w\,\,\,\text{is transitive}\}$ is a dense $G_{\delta}$ in $Z$.
\end{lemma}

\begin{proof}
Write $U=\cap_{x,y\in X}U_{x,y}$, where $U_{x,y}=\{w\in Z\,:\,\exists g\in\Gamma,\,\,\pi_w(g)x=y\}$. It suffices to show that $U_{x,y}$ is open and dense in $Z$ for all $x,y\in X$. It is obvious that $U_{x,y}$ is open. Let us show that $U_{x,y}$ is dense. Let $w\in Z\setminus U_{x,y}$ and $F\subset\ X$ a finite subset. It suffices to construct $\gamma\in Z$ and $g\in \Gamma$ such that $\gamma|_F=w|_F$ and $\pi_{\gamma}(g)x=y$. Since the action has infinite orbits, there exists $h_0, h_1\in H$ such that $h_0y\notin w(\Sigma F)$ and $h_1 x\notin\Sigma F$. Observe that $\Sigma h_1x\cap\Sigma w^{-1}h_0y=\emptyset$. Indeed, if we have $\Sigma h_1x=\Sigma w^{-1}h_0y$ then, for some $\sigma\in\Sigma$, we have $\sigma h_1x=w^{-1}h_0y$. Hence, $\pi_w(g)x=y$ with $g=h_0^{-1}t\sigma h_1$, a contradiction. Define $Y=\Sigma h_1x\sqcup\Sigma w^{-1}h_0y$. Then, $F\subset Y^c$ and $w(Y)=\theta(\Sigma)wh_1 x\sqcup\theta(\Sigma) h_0y$. Define a bijection $\gamma\in S(X)$ by $\gamma|_{Y^c}=w|_{Y^c}$ and
$$\gamma(\sigma h_1 x)=\theta(\sigma) h_0y\quad\text{and}\quad\gamma(\sigma w^{-1} h_0y)=\theta(\sigma) wh_1x\quad\text{for all}\,\,\,\sigma\in\Sigma.$$
By construction, $\gamma\in Z$,  $\gamma|_F=w|_F$ and $\pi_{\gamma}(h_0^{-1}th_1) x=y$.
\end{proof}

Next, we give a sufficient condition for the set of $w\in Z$ for which $\pi_w$ is amenable to be a dense $G_{\delta}$.

\begin{lemma}\label{Folner}
If the action $H\curvearrowright X$ admits a F{\o}lner sequence $(C_n)$ such that $|C_n|\rightarrow\infty$ then the set
$$V=\{w\in Z\,:\,\pi_w\,\,\text{is amenable}\}$$
is a dense $G_{\delta}$ in $Z$.
\end{lemma}

\begin{proof}
At first we prove the following claim.

\textbf{Claim.} Let $Y,F\subset X$ be finite subsets and $w\in Z$ such that $\theta(\Sigma)Y\cap w(F)=\emptyset$. There exists $\gamma\in Z$ such that $\Sigma Y\cap\Sigma\gamma^{-1} Y=\emptyset$ and $\gamma|_F=w|_F$.

\textit{Proof of the Claim.} Write $\theta(\Sigma) Y=\sqcup_{i=1}^n\theta(\Sigma) y_i$. Since the set $\Sigma Y\cup\Sigma F\cup\Sigma w^{-1}(Y)$ is finite and $\Sigma$-invariant, we can find $n$ disjoint $\Sigma$-orbits $\Sigma z_1,\ldots,\Sigma z_n$ in its complement.

Define $\widetilde{F}=\sqcup_{j=1}^n(\Sigma z_j\sqcup\Sigma w^{-1} y_j)$. Then $F\subset\widetilde{F}^c$ and $w(\widetilde{F})=\sqcup_{j=1}^n(\theta(\Sigma)wz_j\sqcup\theta(\Sigma) y_j)$. By the freeness assumption, we can define a bijection $\gamma\in S(X)$ by $\gamma|_{\widetilde{F}^c}=w|_{\widetilde{F}^c}$ and,
$$\gamma(\sigma z_j)=\theta(\sigma) y_j\quad\text{and}\quad \gamma(\sigma w^{-1}y_j)=\theta(\sigma) wz_j\quad\sigma\in\Sigma,\,\,\,1\leq j\leq n.$$
By construction $\gamma\in Z$ and $\gamma|_F=w|_F$. Moreover, $\Sigma Y\cap\Sigma\gamma^{-1}  Y=\sqcup_{j=1}^n\Sigma Y\cap\Sigma\gamma^{-1}(y_j)=\sqcup_{j}\Sigma Y\cap\Sigma z_j=\emptyset$.\hfill{$\Box$}

\textit{End of the proof of Lemma \ref{Folner}.} Write $H=\cup^{\uparrow} F_m$ where $(F_m)$ is an increasing sequence of finite subsets. Since $\Gamma$ is generated by $H$ and $t$ it follows that $\pi_w$ is amenable if and only if, for all $m\geq 1$ there exists a non-empty finite set $C\subset X$ such that
$${\rm Sup}_{g\in F_m\cup\{w\}}\frac{|gC\Delta C|}{|C|}<\frac{1}{m}.$$
Write $V=\cap_{m\geq 1} V_m$ where
$$V_m=\left\{w\in Z\,:\,\,\exists C\subset X\,\,\text{such that}\,\,0<|C|<\infty\,\,\text{and}\,\,{\rm Sup}_{g\in F_m\cup\{w\}}\frac{|gC\Delta C|}{|C|}<\frac{1}{m}\right\}.$$
It is easy to see that $V_m$ is open. Let us show that $V_m$ is dense. Let $w\in Z$ and $F\subset X$ a finite subset. Observe that, for any F{\o}lner sequence $(C_n)$ and any finite subset $K\subset H$, the sequence $(KC_n)$ is again a F{\o}lner sequence. Also, if $|C_n|\rightarrow\infty$ then, for all  finite subset $F\subset X$, there exists $n_0\in\N$ such that the sequence $(C_n\setminus F)_{n\geq n_0}$ is a F{\o}lner sequence. Hence, up to a shifting of the indices, the sequence $(D_n)$, where $D_n=\Sigma C_n\setminus(\Sigma F\cup\Sigma w(\Sigma F))$ is a $\Sigma$-globally invariant F{\o}lner sequence such that $D_n\cap F=\emptyset$ and $\theta(\Sigma)D_n\cap w(F)=\emptyset$ for all $n\in\N$. Let $N\in\N$ large enough such that
$$|gD_N\Delta D_N|<\frac{|D_N|}{2m|\Sigma|}\quad\text{for all}\quad g\in F_m\cup\theta(\Sigma).$$
By the claim, there exists $\gamma_0\in Z$ such that $\gamma_0|_F=w|_F$ and $D_N\cap\Sigma\gamma_0^{-1}(D_N)=\emptyset$. Write
$$D_N=\sqcup_{i=1}^L\Sigma x_i\quad\text{and}\quad \theta(\Sigma)D_N=\sqcup_{i=1}^K\theta(\Sigma)y_i,$$
where $K\geq L$. Observe that the sets $\Sigma x_i$ and $\Sigma \gamma_0^{-1}y_j$ are pairwise disjoint. Define $Y=\bigsqcup_{i=1}^L\Sigma x_i\sqcup\Sigma \gamma_0^{-1} y_i$. Observe that $F\subset Y^c$ and $\gamma_0(Y)=\bigsqcup_{i=1}^L\theta(\Sigma) \gamma_0x_i\sqcup\theta(\Sigma) y_i$. Define a bijection $\gamma\in S(X)$ by $\gamma|_{Y^c}=\gamma_0|_{Y^c}$ and,
$$\gamma(\sigma x_i)=\theta(\sigma) y_i\quad\text{and}\quad \gamma(\sigma \gamma_0^{-1} y_i)=\theta(\sigma)\gamma_0 x_i\quad \sigma\in\Sigma,\,\, 1\leq i\leq L.$$
By construction, $\gamma\in Z$ and $\gamma|_F=\gamma_0|_F=w|_F$. Observe that
$$|\theta(\Sigma)D_N\Delta\gamma(D_N)|=|\theta(\Sigma)D_N|-|\gamma(D_N)|=|\theta(\Sigma)D_N|-|D_N|=|\theta(\Sigma)D_N\Delta D_N|.$$
Moreover,
$$|D_N\Delta \theta(\Sigma)D_N|\leq\sum_{\sigma\in\Sigma}|D_N\Delta \theta(\sigma)D_N|<\frac{|D_{N}|}{2m}.$$
It follows that
$$|D_N\Delta\gamma(D_N)|\leq|D_N\Delta \theta(\Sigma)D_N|+|\theta(\Sigma)D_N\Delta\gamma(D_N)|<\frac{|D_{N}|}{m}.$$
We also have $|D_N\Delta gD_N|<\frac{|D_N|}{m}$ for all $g\in F_m$ hence, it suffices to define $C=D_N$ to see that $\gamma\in V_m$.
\end{proof}

We can now prove Theorem \ref{THMA}.1. It suffices to prove that, under suitable conditions, the set of $w\in Z$ such that $\pi_{w}$ is faithful is a dense $G_{\delta}$ in $Z$. It is obvious that such a set is a $G_{\delta}$. The difficulty will be to prove that it is dense. The idea is very simple: we will start from a faithful action of $\Gamma$ on $Y$ and consider the faithful action of $H$ obtained by restricting the faithful action of $\Gamma$ on infinitely many copies of $Y$. For this action of $H$, any $w\in Z$ can be approximate by a $\gamma\in Z$ such that $\pi_{\gamma}$ is faithful by taking roughly $\gamma$ equals to $w$ on sufficiently many (but finitely) copies of $Y$ and $\gamma$ equals to the original stable letter $t\in\Gamma$ on the remaining copies of $Y$ to insure that $\pi_{\gamma}$ is faithful. Let us write this argument precisely.

\textit{Proof of Theorem \ref{THMA}.1.} Suppose that $H$ admits an amenable and faithful action with infinite orbits and free on $\Sigma$ and $\theta(\Sigma)$. Define $\Gamma={\rm HNN}(H,\Sigma,\theta)=\langle H,t\rangle$. By Lemma \ref{induction} and example \ref{example}, there exists a faithful action $\Gamma\curvearrowright Y$ with infinite $H$-orbits such that $\Sigma,\theta(\Sigma)\curvearrowright Y$ are free and the action $H\curvearrowright Y$ is amenable. Consider the faithful action $\Gamma\curvearrowright X$ with $X=Y\times\N$ given by $g(y,n)=(gy,n)$ for $g\in\Gamma$ and $(y,n)\in X$. We view $H<\Gamma<S(X)$. It is obvious that the actions $\Sigma,\theta(\Sigma)\curvearrowright X$ are free, the $H$-orbits are infinite and the action $H\curvearrowright X$ is amenable. Hence, by Lemma \ref{InfiniteFolner}, there exists a F{\o}lner sequence $(C_n)$ for the $H$-action whose size goes to infinity (one could also use Lemma \ref{LemSupp}). Thus, we can apply Lemmas \ref{transitif} and \ref{Folner} to the action $H\curvearrowright X$. Hence, it suffices to show that the set $O=\{w\in Z\,:\,\pi_w\,\,\text{is faithful}\}$ is a dense $G_{\delta}$ in $Z$. Writing $O=\cap_{g\in\Gamma\setminus\{1\}}O_g$, where $O_g=\{w\in Z\,:\,\pi_w(g)\neq\id\}$ is obviously open, it suffices to show that $O$ is dense. Write $X=\cup^{\uparrow}X_n$ where $X_n=\{(x,k)\in X\,:\,k\leq n\}$ is infinite and globally invariant under $\Gamma$ for the original action. Let $w\in Z$ and $F\subset X$ a finite subset. Let $N\in\N$ large enough such that $\Sigma F\cup w(\Sigma F)\subset X_N$. The set $X_N\setminus\Sigma F$ (resp. $X_N\setminus w(\Sigma F)$) is infinite and globally invariant under $\Sigma$ (resp. $\theta(\Sigma)$). Hence, there exists a bijection $\gamma_0\,:\,X_N\setminus\Sigma F\rightarrow X_N\setminus w(\Sigma F)$ satisfying $\gamma_0\sigma=\theta(\sigma)\gamma_0$ for all $\sigma\in\Sigma$. Define $\gamma\in S(X)$ by $\gamma|_{\Sigma F}=w|_{\Sigma F}$, $\gamma|_{X_N\setminus\Sigma F}=\gamma_0$ and $\gamma|_{X_N^c}=t|_{X_N^c}$. By construction, $\gamma\in Z$ and $\gamma|_F=w|_F$. Moreover, since $\pi_{\gamma}(g)(y,n)=(gy,n)$ for all $n>N$ and since $\Gamma\curvearrowright Y$ is faithful, it follows that $\pi_{\gamma}$ is faithful.\hfill{$\Box$}

\begin{remark}\label{rmkhnn}
The following more general result is actually true.

\textit{For all amenable and faithful action on a countable set $H\curvearrowright Y$ with infinite orbits and free on $\Sigma$ and $\theta(\Sigma)$, there exists an amenable, transitive and faithful action on a countable set $\Gamma\curvearrowright X$ such that, for all $h\in H$, ${\rm Fix}_Y(h)=\emptyset$ implies ${\rm Fix}_X(h)=\emptyset$.}

Indeed, it follows from Lemma \ref{induction} and Example \ref{example}, that the replacement by the induced action preserves the property that elements in $H$ have an empty fixed point set. Also, the replacement by the action on $X=Y\times\N$ preserves this property. Since $\pi_w(h)=h$ for all $h\in H$ and all $w\in Z$, this proves the remark.
\end{remark}

\section{Amalgamated free products in the class $\mathcal{A}$}

In this section we prove Theorem \ref{THMA}.2. The notations are independent of the ones of section $3$.

Let $X$ be an infinite countable set. For the rest of this section we assume that $\Gamma_1,\Gamma_2< S(X)$ are two countable subgroups of the Polish group of bijections of $X$ with a common finite subgroup $\Sigma$ such that $\Sigma\curvearrowright X$ is free. Define
$$Z=\{w\in S(X)\,:\,w\sigma=\sigma w\,\,\,\text{for all}\,\,\sigma\in\Sigma\}.$$
$Z$ is a non-trivial closed subgroup of $S(X)$. Let $\Gamma=\Gamma_1*_{\Sigma}\Gamma_2$. By the universal property, for all $w\in Z$, there exists a unique group homorphism $\pi_w\,:\,\Gamma\rightarrow S(X)$ such that $\pi_w(g)=g$ and $\pi_w(h)=w^{-1}hw$ for all $g\in\Gamma_1,h\in\Gamma_2$.

\begin{lemma}\label{lemtransfree}
If the actions $\Gamma_1\curvearrowright X$ and $\Gamma_2\curvearrowright X$ have infinite orbits then the set

$$U=\{w\in Z\,:\,\pi_w\,\,\text{is transitive}\,\}$$
is a dense $G_{\delta}$ in $Z$.
\end{lemma}

\textit{Proof.} Write $U=\cap_{x,y\in X} U_{x,y}$, where $U_{x,y}=\{w\in Z\,:\,\,\text{there exists}\,\,g\in\Gamma\,\,\text{such that}\,\,\pi_w(g)x=y\}$. Since $U_{x,y}$ is open in $Z$ for all $x,y\in X$, it suffices to show that it is dense in $Z$. Let $x,y\in X$, $w\in Z$ and $F\subset X$ a finite subset. Since $\Gamma_1\curvearrowright X$ has infinite orbits, there exists $g_1\in\Gamma_1$ such that $g_1x\notin\Sigma F$. By the same argument, there exists also $g_2\in\Gamma_1$ such that $g_2^{-1}y\notin\Sigma F\cup\Sigma g_1x$. Take $z,t\in X$ in the same $\Gamma_2$-orbit and in the complement of the finite set $w(\Sigma F\cup\Sigma g_1 x\cup\Sigma g_2^{-1} y)$ such that $\Sigma z\cap\Sigma t=\emptyset$. Write $t=hz$ where $h\in\Gamma_2$ and define 
$$Y=\Sigma g_1 x\sqcup\Sigma g_2^{-1}y\sqcup\Sigma w^{-1}z\sqcup\Sigma w^{-1}t.$$
One has $F\subset Y^c$ and $w(Y)=\Sigma wg_1 x\sqcup\Sigma wg_2^{-1}y\sqcup\Sigma z\sqcup\Sigma t$. Define $\gamma\in S(X)$ by $\gamma|_{Y^c}=w|_{Y^c}$ and,
$$\gamma(\sigma g_1x)=\sigma z,\,\,\,\gamma(\sigma g_2^{-1}y)=\sigma t,\,\,\,\gamma(\sigma w^{-1}z)=\sigma wg_1 x,\,\,\,\gamma(\sigma w^{-1}t)=\sigma wg_2^{-1} y\quad\forall\sigma\in\Sigma.$$
By construction, $\gamma\in Z$ and $\gamma|_F=w|_F$. Moreover, with $g=g_2hg_1\in\Gamma$, one has
$$\pi_{\gamma}(g)x=g_2\gamma^{-1}h\gamma g_1x=g_2\gamma^{-1}hz=g_2\gamma^{-1}t=g_2g_2^{-1}y=y.\hfill{\Box}$$

\begin{lemma}\label{LemFolFree}
If there exist F{\o}lner sequences $(C_n)$ and $(D_n)$ for the actions $\Gamma_1\curvearrowright X$ and $\Gamma_2\curvearrowright X$ respectively such that $|C_n|,|D_n|\rightarrow\infty$ and $\frac{|D_n|}{|C_n|}\rightarrow 1$ then the set $V=\{w\in Z\,:\,\pi_w\,\,\text{is amenable}\,\,\}$ is a dense $G_{\delta}$ in $Z$.
\end{lemma}

\begin{proof}
We start the proof with the following simple claim.

\textbf{Claim.}\begin{enumerate}
\item For all finite subsets $Y_1,Y_2\subset X$ there exists  $\Sigma$-invariant F{\o}lner sequences $(C'_n)$ and $(D'_n)$ for the actions $\Gamma_1\curvearrowright X$ and $\Gamma_2\curvearrowright X$ respectively such that $|C'_n|,|D'_n|\rightarrow\infty$, $\frac{|D'_n|}{|C'_n|}\rightarrow 1$ and $C'_n\cap Y_1=\emptyset$, $D'_n\cap Y_2=\emptyset$ for all $n\in\N$.
\item Let $F,Y_1\subset X$ be finite subsets such that $\Sigma Y_1\cap F=\emptyset$. For all finite subset $Y_2\subset X$ and all $w\in Z$, there exists $\gamma\in Z$ such that $\gamma|_F=w|_F$ and  $ \Sigma Y_1\cap\gamma^{-1}( \Sigma Y_2)=\emptyset$.
\end{enumerate}
\textit{Proof of the claim.}
$1$. Take $C_n'=\Sigma C_n\setminus\Sigma Y_1$ and $D_n'=\Sigma D_n\setminus \Sigma Y_2$ and shift the indices if necessary. We leave the details to the reader. 

$2$. Write $\Sigma Y_1=\sqcup_{i=1}^l\Sigma x_i$. Since the set $\Sigma wY_1\cup\Sigma wF\cup\Sigma Y_2$ is finite and $\Sigma$ invariant, we can find $l$ disjoint $\Sigma$-orbits $\Sigma z_1,\ldots \Sigma z_l$ in its complement. Define $Y=\sqcup_{i=1}^l\Sigma x_i\sqcup\Sigma w^{-1} z_i$. Then, $F\subset Y^c$ and $w(Y)=\sqcup_{i=1}^l\Sigma wx_i\sqcup\Sigma  z_i$. Define $\gamma\in S(X)$ by $\gamma|_{Y^c}=w|_{Y^c}$ and,
$$\gamma(\sigma x_i)=\sigma z_i\quad\gamma(\sigma w^{-1} z_i)=\sigma w x_i\quad\text{for all}\,\,\,\sigma\in\Sigma,\,\,1\leq i\leq l.$$
By construction $\gamma\in Z$ and $\gamma|_{F}=w|_{F}$. Moreover, $\gamma(\Sigma Y_1)\cap\Sigma Y_2=\emptyset$.\hfill{$\Box$}

\textit{End of the proof of Lemma \ref{LemFolFree}.} Write $\Gamma_1=\cup^{\uparrow}F_m$ and $\Gamma_2=\cup^{\uparrow}G_m$, where $|F_m|, |G_m|<\infty$. Since $\Gamma$ is generated by $\Gamma_1$ and $\Gamma_2$ it follows that $\pi_w$ is amenable if and only if, for all $m\geq 1$ there exists a non-empty finite set $C\subset X$ such that 
$$\frac{|\pi_w(g)C\Delta C|}{|C|}=\frac{|gC\Delta C|}{|C|}<\frac{1}{m}\,\,\,\,\forall g\in F_m\quad\text{and}\quad
\frac{|\pi_w(h)C\Delta C|}{|C|}=\frac{|h w( C ) \Delta w ( C ) |}{|C|}<\frac{1}{m}\,\,\,\,\forall h\in G_m.$$
Write $V=\cap_{m\geq 1} V_m$ where
$$V_m=\{w\in Z\,:\,\exists C\subset X,\,0<|C|<\infty,\,\,\text{such that}\,\,\frac{|gC\Delta C|}{|C|}<\frac{1}{m}\,\,\text{and}\,\,\frac{|h w( C ) \Delta w ( C ) |}{|C|}<\frac{1}{m}\,\,\forall g\in F_m, h\in G_m\}.$$
Since $V_m$ is open in $Z$, it suffices to show that $V_m$ is dense in $Z$. Let $w\in Z$ and $F\subset X$ a finite subset. By the first assertion of the claim, we can assume that $(C_n)$ and $(D_n)$ are $\Sigma$-invariant F{\o}lner sequences such that $C_n\cap F=\emptyset$ and $D_n\cap w(F)=\emptyset$ for all $n\in\N$. Let $N\in\N$ large enough such that
$$\frac{|gC_N\Delta C_N|}{|C_N|}<\frac{1}{m},\quad\frac{|hD_N\Delta D_N|}{|D_N|}<\frac{1}{4m},\quad\left|1-\frac{|D_N|}{|C_N|}\right|<\frac{1}{4m}\quad\text{for all}\,\,g\in F_m,\,\,h\in G_m.$$

By the second assertion of the claim we may and will assume that $C_N\cap w^{-1}(D_N)=\emptyset$. Write $C_N=\sqcup_{i=1}^l\Sigma x_i$ and $D_N=\sqcup_{j=1}^k\Sigma y_j$. Let $M={\rm Min}(l,k)$ and define $Y=\sqcup_{i=1}^M\Sigma x_i\sqcup\Sigma w^{-1}y_i$. Then, one has $F\subset Y^c$ and $w(Y)=  \sqcup_{i=1}^M\Sigma wx_i\sqcup\Sigma y_i$. Define $\gamma\in S(X)$ by $\gamma|_{Y^c}=w|_{Y^c}$ and,
$$\gamma(\sigma x_i)=\sigma y_i\quad\gamma(\sigma w^{-1}y_i)=\sigma wx_i\quad\text{for all}\,\,\sigma\in\Sigma,\,\,1\leq i\leq M.$$
By construction $\gamma\in Z$ and $\gamma|_F=w|_F$. Moreover, $|\gamma(C_N)\Delta D_N|=\left| \,|D_N|-|C_N| \,\right|<\frac{|C_N|}{4m}$. Hence, for all $h\in G_m$, one has,
\begin{eqnarray*}
|h\gamma(C_N)\Delta\gamma(C_N)|&\leq& |h\gamma(C_N)\Delta hD_N|+|hD_N\Delta D_N|+|D_N\Delta \gamma(C_N)|
=2|\gamma(C_N)\Delta D_N|+|hD_N\Delta D_N|\\
&<&\frac{|C_N|}{2m}+\left(1+\frac{1}{4m}\right)\frac{|hD_N\Delta D_N|}{|D_N|}|C_N|<\frac{|C_N|}{m}.
\end{eqnarray*}                                                                                                            
Defining $C=C_N$, we see that $\gamma\in V_m$.
\end{proof}

\textit{Proof of Theorem \ref{THMA}.2.}
Suppose that the triple $(\Sigma,\Gamma_1,\Gamma_2)$ satisfies the hypothesis of Theorem \ref{THMA}.2 and define $\Gamma=\Gamma_1*_{\Sigma}\Gamma_2$. By Lemma \ref{induction}, for all $i\in\{1,2\}$, there exists a faithful action $\Gamma\curvearrowright Y_i$ with infinite $\Gamma_i$-orbits such that $\Sigma\curvearrowright Y_i$ is free and the action $\Gamma_i\curvearrowright Y_i$ is amenable. Moreover, by Example \ref{example}, $\Gamma\curvearrowright Y_i$ also has infinite $\Gamma_j$-orbits for $j\neq i$. Define $Y=Y_1\sqcup Y_2$. Then the natural faithful action $\Gamma\curvearrowright Y$ has infinite $\Gamma_i$-orbits for $i=1,2$, $\Sigma\curvearrowright Y$ is free and $\Gamma_i\curvearrowright Y$ is amenable for $i=1,2$. Define $X=Y\times\N$ with the faithful $\Gamma$-action given by $g(y,n)=(gy,n)$ for $g\in\Gamma$ and $(y,n)\in X$. View $\Sigma<\Gamma_1,\Gamma_2<\Gamma<S(X)$. It is clear that $\Sigma$ acts freely on $X$. Moreover, by Lemma \ref{LemSupp}, we can find F{\o}lner sequences $(C_n)$ and $(D_n)$ for the actions of $\Gamma_1$ and $\Gamma_2$ on $X$ respectively such that $|C_n|\rightarrow\infty$, $|D_n|\rightarrow\infty$ and $\frac{|D_n|}{|C_n|}\rightarrow 1$. Thus, we can apply Lemmas \ref{lemtransfree} and \ref{LemFolFree} to the actions $\Gamma_1,\Gamma_2\curvearrowright X$. Hence, it suffices to show that the set $O=\{w\in Z\,:\,\pi_w\,\,\,\text{is  faithful}\}$ is a dense $G_{\delta}$ in $Z$. As in the proof of Theorem \ref{THMA}.1 it is easy to write $O$ as a countable intersection of open sets. Hence, it suffices to show that $O$ is dense in $Z$. Let $w\in Z$ and $F\subset X$ a finite subset. Write $X=\cup^{\uparrow}X_n$ where $X_n=\{(x,k)\in X\,:\,k\leq n\}$ is infinite globally invariant under $\Gamma$. Let $N\in\N$ large enough such that $\Sigma F\cup w(\Sigma F)\subset X_N$. Since the sets $X_N\setminus\Sigma F$ and $X_N\setminus w(\Sigma F)$ are infinite and $\Sigma$-invariant, there exists a bijection $\gamma_0\,:\,X_N\setminus\Sigma F\rightarrow X_N\setminus w(\Sigma F)$ such that $\gamma_0\sigma=\sigma\gamma_0$ for all $\sigma\in\Sigma$. Define $\gamma\in S(X)$ by $\gamma|_{\Sigma F}=w|_{\Sigma F}$, $\gamma|_{X_N\setminus\Sigma F}=\gamma_0$ and $\gamma|_{X_N^c}=\id|_{X_N^c}$. By construction, $\gamma\in Z$ and $\gamma|_{F}=w|_{ F}$. Moreover, since $\pi_{\gamma}(g)(y,n)=(gy,n)$ for all $g\in\Gamma$ and all $(y,n)\in X$ with $n>N$ and because $\Gamma\curvearrowright Y$ is faithful, it follows that $\pi_{\gamma}$ is faithful.\hfill{$\Box$}

\begin{remark}\label{rmkfree}
The following more general result is actually true.

\textit{If, for $i=1,2$, there exists an amenable and faithful action on a countable set $\Gamma_i\curvearrowright Y_i$ with infinite orbits and free on $\Sigma$ then, there exists an amenable, transitive and faithful action on a countable set $\Gamma\curvearrowright X$ with the property that, for all $i=1,2$, for all $h\in\Gamma_i$, ${\rm Fix}_{Y_i}(h)=\emptyset$ implies ${\rm Fix}_X(h)=\emptyset$.}

Indeed, the first replacement by the induced action from $\Gamma_i$ to $\Gamma$ preserves the property that the elements in $\Gamma_i$ have an empty fixed point set. Moreover, Lemma \ref{induction} and example \ref{example} imply that $\Gamma_i\curvearrowright Y_j$ is free for $i\neq j$. Hence, the property to have an empty fixed point set for the actions $\Gamma_1,\Gamma_2\curvearrowright Y=Y_1\sqcup Y_2$ is preserved. The replacement by the action on $X=Y\times\N$ also preserves this property. Since, for all $w\in Z$, $\pi_w(g)=g$ for all $g\in\Gamma_1$ and ${\rm Fix}_X(\pi_w(h))=w^{-1}({\rm Fix}_X(h))$ for all $h\in\Gamma_2$, this proves the remark.
\end{remark}


\section{Groups acting on trees in the class $\mathcal{A}$}


This section contains the proof of Theorem \ref{THMB}. Let $\Gr$ be a graph. We denote by $\EG$ its edge set and by $\VG$ its vertex set. For $e\in\EG$ we denote by $s(e)$ the source of $e$ and $r(e)$ the range of $e$.

Let $\Gamma$ be a countable group acting without inversion on a non-trivial tree ${\rm T}$ with finite quotient graph $\mathcal{G}={\rm T}/\Gamma$ and finite edge stabilizers. By \cite{Se83}, the quotient graph $\mathcal{G}$ can be equipped with the structure of a graph of groups $(\mathcal{G},\{\Gamma_p\}_{p\in\VG},\{\Sigma_e\}_{e\in\EG})$ where each $\Sigma_e$ is isomorphic to an edge stabilizer and each $\Gamma_p$ is isomorphic to a vertex stabilizer and such that $\Gamma$ is the fundamental group of this graph of groups i.e., given a fixed maximal subtree $\mathcal{T}\subset\mathcal{G}$, $\Gamma$ is generated by the groups $\Gamma_p$ for $p\in\VG$ and the edges $e\in\EG$ with the relations
$$\overline{e}=e^{-1},\quad s_{e}(x)=er_{e}(x)e^{-1}\,\,,\,\forall x\in\Sigma_e\quad\text{and}\quad e=1\,\,\,\,\forall e\in {\rm E}(\mathcal{T}),$$
where $s_e\,:\,\Sigma_e\rightarrow \Gamma_{s(e)}$ and $r_e\,:\,\Sigma_e\rightarrow\Gamma_{r(e)}$ are respectively the source and range group homomorphisms.  We will prove the following stronger version of Theorem \ref{THMB} by induction on $n=\frac{1}{2}|E(\mathcal{G})|\geq 1$.

\begin{theorem}\label{THMB2}
Suppose that, for all $p\in\VG$, there exists an amenable and faithful action on a countable set $\Gamma_p\curvearrowright X_p$ with infinite orbits and free on $s_e(\Sigma_e)$ for all $e\in\EG$ such that $s(e)=p$. Then, there exists an amenable, faithful and transitive action on a countable set $\Gamma\curvearrowright X$ such that, for all $p\in\VG$ and all $h\in\Gamma_p$, ${\rm Fix}_{X_p}(h)=\emptyset$ implies  ${\rm Fix}_{X}(h)=\emptyset$.
\end{theorem}

\begin{proof}
If $n=1$ then $\Gamma$ is either an amalgamated free product $\Gamma=\Gamma_1*_{\Sigma}\Gamma_2$ where $(\Sigma,\Gamma_1,\Gamma_2)$ satisfies the hypothesis of Theorem \ref{THMA}.2 or an HNN-extension $\Gamma={\rm HNN}(H,\Sigma,\theta)$ where $(H,\Sigma,\theta)$ satisfies the hypothesis of Theorem \ref{THMA}.1. In the amalgamated free product case we use Remark \ref{rmkfree} and in the HNN-extension case we use Remark \ref{rmkhnn} to obtain that $\Gamma$ satisfies the conclusion of the theorem. Let $n\geq 1$ and suppose that the conclusion holds for all $1\leq k\leq n$. Suppose that $\frac{1}{2}|E(\mathcal{G})|=n+1$. Let $e\in\EG$ and let  $\mathcal{G}'$ be the graph obtained from $\mathcal{G}$ by removing the edges $e$ and $\overline{e}$.

\textbf{Case 1: $\mathcal{G}'$ is connected.} It follows from Bass-Serre theory that $\Gamma={\rm HNN}(H,\Sigma,\theta)$ where $H$ is the fundamental group of our graph of groups restricted to $\mathcal{G}'$, $\Sigma=r_e(\Sigma_e)<H$ and $\theta\,:\,\Sigma\rightarrow H$ is given by $\theta=s_e\circ r_e^{-1}$. By the induction hypothesis and Remark \ref{rmkhnn} it follows that $\Gamma$ satisfies the conclusion of the theorem.

\textbf{Case 2: $\mathcal{G}'$ is not connected.} Let $\mathcal{G}_1$ and $\mathcal{G}_2$ be the two connected components of $\mathcal{G}'$ such that $s(e)\in{\rm V}(\mathcal{G}_1)$ and $r(e)\in{\rm V}(\mathcal{G}_2)$. Bass-Serre theory implies that $\Gamma=\Gamma_1*_{\Sigma}\Gamma_2$, where $\Gamma_i$ is the fundamental group of our graph of groups restricted to $\mathcal{G}_i$, $i=1,2$, and $\Sigma=\Sigma_e$ is viewed as a subgroup of $\Gamma_1$ via the map $s_e$ and as a subgroup of $\Gamma_2$ via the map $r_e$. By the induction hypothesis and Remark \ref{rmkfree}, $\Gamma$ satisfies the conclusion of the theorem.
\end{proof}

The proof of Theorem \ref{THMB} follows from Theorem \ref{THMB2} and Lemma \ref{AFtoF} since an almost free action on an infinite set is faithful.

\end{document}